\documentclass[11pt]{amsart}

\usepackage[top=1in, bottom=1in, left=1in, right=1in]{geometry}
\usepackage[dvipsnames]{xcolor}
\usepackage[colorlinks=true, pdfstartview=FitV, linkcolor=RoyalBlue,citecolor=ForestGreen, urlcolor=blue]{hyperref}
\usepackage{xfrac}

\usepackage[english]{babel}
\usepackage{graphicx} 
\usepackage{amsthm} 
\usepackage{amsmath}
\usepackage{amssymb}
\usepackage{amsthm} 
\usepackage{mathrsfs} 
\usepackage{bbm} 
\usepackage{enumerate} 
\usepackage[shortlabels]{enumitem}
\usepackage{xcolor}
\usepackage{nicefrac}
\usepackage{subfigure}

\usepackage{hyperref}

\newcommand{\N}{\mathbb{N}}
\newcommand{\Z}{\mathbb{Z}}

\newcommand{\R}{\mathbb{R}}
\newcommand{\C}{\mathbb{C}}
\newcommand{\T}{\mathbb{T}}

\newcommand{\dif}{\,\mathrm{d}}

\newtheorem{lemma}{Lemma}[section]
\newtheorem{prop}[lemma]{Proposition}
\newtheorem{thm}[lemma]{Theorem}
\newtheorem{defi}[lemma]{Definition}

\newtheorem{Rem}[lemma]{Remark}

\title{Sharp nonuniqueness for the forced 2D Navier-Stokes and dissipative SQG equations}
\author{Francisco Mengual and Marcos Solera}
\date{\today}

\begin{document}
	
\begin{abstract}
We prove a sharp nonuniqueness result for the forced generalized SQG equation. First, this yields nonunique $\dot{H}^s$-energy solutions below the Miura–Ju class. In particular, this shows that the solutions constructed by Resnick and Marchand for the dissipative SQG equation are not necessarily unique. Second, this establishes nonuniqueness below the Ladyzhenskaya–Prodi–Serrin class for the 2D Navier–Stokes equation, as well as below the Constantin–Wu and Dong–Chen–Zhao–Liu classes for the dissipative SQG equation.
\end{abstract}

\maketitle

\setcounter{tocdepth}{1} 
\tableofcontents

\section{Introduction and main results}
 
In this paper, we apply Vishik’s approach \cite{Vishikpp1,Vishikpp2,ABCDGMKpp,CFMSEuler}, together with our previous work \cite{CFMSSQG} with Castro and Faraco, and Golovkin’s trick \cite{Golovkin64,DolceMescolinipp}, to address the nonuniqueness for the forced 2D Navier–Stokes equation and the forced dissipative Surface Quasi-Geostrophic (SQG) equation.

Both systems can be written as particular cases of the generalized SQG equation with fractional dissipation (with $\Lambda = (-\Delta)^{1/2}$):
\begin{subequations}\label{eq:SQG}
\begin{equation}\label{eq:SQG:1}
\partial_t\theta + v\cdot\nabla\theta+\Lambda^\beta\theta=f,
\end{equation}
posed on $[0,\infty)\times\R^2$, for some given external force $f(t,x)$ and initial datum $\theta^\circ(x)$
\begin{equation}\label{eq:SQG:0}
\theta|_{t=0}
=\theta^\circ.
\end{equation}
The velocity $v(t,x)$ is recovered from $\theta(t,x)$ through the \emph{$\alpha$-Biot-Savart law}
\begin{equation}\label{eq:SQG:BS}
v=-\nabla^\perp\Lambda^{\alpha-2}\theta.
\end{equation}
\end{subequations}
We refer to the system \eqref{eq:SQG} as the \emph{$(\alpha,\beta)$-SQG equation}, for the range of parameters
$$
0\leq\alpha\leq 1,
\qquad
0<\beta<3+\alpha.
$$

We note that the 2D Navier–Stokes equation corresponds to the choice $\alpha = 0$. In this case, \eqref{eq:SQG:BS} reduces to the standard Biot–Savart law, and therefore $\theta = \nabla^\perp \cdot v$ represents the vorticity $\omega$. The standard 2D Navier–Stokes equation corresponds to $\beta = 2$, while the regimes $\beta < 2$ and $\beta > 2$ are usually referred to as the hypo-dissipative and hyper-dissipative cases, respectively.

The dissipative SQG equation corresponds to $\alpha = 1$. The intermediate range $0 < \alpha < 1$ interpolates between Navier–Stokes and SQG and is commonly referred to as the generalized SQG equation.
In addition, we refer to the case without diffusion as the \emph{$\alpha$-SQG equation}
\begin{equation}\label{eq:SQG:nodif}
\partial_t \theta + v \cdot \nabla \theta = f,
\end{equation}
coupled with the initial condition \eqref{eq:SQG:0}, and the $\alpha$–Biot–Savart law \eqref{eq:SQG:BS}. In the absence of diffusion, the extreme cases $\alpha = 0$ and $\alpha = 1$ correspond to the 2D Euler and SQG equations, respectively.

In the recent groundbreaking works \cite{Vishikpp1,Vishikpp2}, Vishik successfully established nonuniqueness for the forced 2D Euler equation below the Yudovich class (see also the clear exposition in \cite{ABCDGMKpp}, as well as \cite{ACM} for the vanishing viscosity limit). The proof is based on the self-similar instability scenario proposed by Jia and \v{S}ver\'{a}k for the 3D Navier–Stokes equation \cite{JiaSverak15}. Remarkably, Albritton, Bru\`e, and Colombo built upon these ideas to construct the first nonunique Leray–Hopf solutions for the forced 3D Navier–Stokes equation \cite{ABC22}. Very recently, Hou, Wang, and Yang completed the first rigorous (computer-assisted) proof in the unforced case \cite{HWYpp}.

In \cite{CFMSSQG}, together with Castro and Faraco, we proved nonuniqueness for the $\alpha$–SQG equation \eqref{eq:SQG:nodif}.
To this end, we first simplified Vishik’s proof in \cite{CFMSEuler} by constructing smooth, compactly supported unstable vortices, and then carefully adapted the argument to the generalized SQG equation without diffusion.

The aim of the present work is to show that the vortices constructed in \cite{CFMSSQG} can also be used to establish nonuniqueness in the diffusive setting. First, we follow Vishik’s spectral argument to treat the diffusion as a perturbation. Second, we apply Golovkin’s trick \cite{Golovkin64}, recently rediscovered by Dolce and Mescolini \cite{DolceMescolinipp}, to bypass the nonlinear instability step.

\subsection{Main result}
We present the main nonuniqueness theorem for the $(\alpha,\beta)$--SQG equation \eqref{eq:SQG}, formulated in the Bochner spaces
$L_t^p \dot{W}^{s,q} = L^p([0,T], \dot{W}^{s,q}(\mathbb{R}^2))$.
Recall that $\theta \in \dot{W}^{s,q}$ means that $\Lambda^{s}\theta \in L^{q}$.
In this first version (Theorem~\ref{thm:main}), we consider the full range of parameters. In the subsequent sections, we derive several interesting corollaries as particular cases, which we divide into two groups.
See Sections~\ref{sec:LH} and \ref{sec:LPS} for precise definitions and references.

\begin{enumerate}
    \item \emph{$\dot{H}^s$-energy solutions}, where the integrability exponents are fixed ($p = \infty$ and $q = 2$) and the regularity exponent $s$ varies:
    \begin{itemize}
    \item Theorem \ref{thm:energysolutions} yields nonunique solutions $\theta\in C_t\dot{H}^s\cap L_t^2\dot{H}^{s+\tiny\frac{\beta}{2}}$ for $s+\beta-\alpha<1$. The case $\alpha=1$ shows sharpness of the Miura--Ju class: $s<2-\beta$.
    \item Theorem \ref{thm:Leray:H} yields nonunique  
    solutions $\theta\in C_t\dot{H}^\frac{\alpha-2}{2}\cap L_t^2\dot{H}^{\frac{\alpha+\beta-2}{2}}$ for $\beta < 2 + \frac{\alpha}{2}$. 
    \begin{itemize}
        \item The case $\alpha=0$ recovers the Leray–Hopf nonuniqueness
$v \in C_t L^2 \cap L_t^2 \dot H^{\frac{\beta}{2}}$
for the hypodissipative 2D Navier–Stokes equation of Albritton and Colombo \cite{AC}.
        \item The case $\alpha=1$ yields nonunique Marchand solutions $\theta\in C_t\dot{H}^{-\frac{1}{2}}\cap L_t^2\dot{H}^{\frac{\beta-1}{2}}$ for the dissipative SQG equation when $\beta<\frac{5}{2}$.    \end{itemize} 
    \item Theorem \ref{thm:Leray:E} yields nonunique  
    solutions $\theta\in C_t L^2\cap L_t^2\dot{H}^{\frac{\beta}{2}}$ for $\beta<1+\alpha$. The case $\alpha=1$ shows nonuniqueness of Resnick solutions when $\beta<2$.
    \end{itemize}
    \vspace*{4pt}
    \item \emph{$L_t^p L^q$-solutions}, where the regularity exponent $s$ is fixed (specifically, $s = -1, 0, 1$) and the admissible ranges of $p$ and $q$ are determined:
    \vspace*{2pt}
    \begin{itemize}
    \item Theorem \ref{thm:LPS} yields nonunique solutions $v \in L_t^p L^q$ to the 2D Navier–Stokes equation ($\alpha = 0, \beta = 2$) below the Ladyzhenskaya–Prodi–Serrin class: $\frac{2}{p} + \frac{2}{q} > 1$.
    \item Theorem \ref{thm:ConstantinWu} yields nonunique solutions $\theta\in L_t^p L^q$ to the dissipative SQG equation ($\alpha=1,\beta<4$) below the Constantin--Wu class: $\frac{\beta}{p}+\frac{2}{q}>\beta-1$.
    \item Theorem \ref{thm:DCZL} yields nonunique solutions $\nabla\theta\in L_t^p L^q$ to the generalized SQG equation ($0\leq\alpha\leq 1,\beta <3+\alpha$) below the Dong--Chen--Zhao--Liu class: $\frac{\beta}{p}+\frac{2}{q}>\beta-\alpha+1$.
    \end{itemize}
\end{enumerate}

\begin{thm}\label{thm:main} 
Let $0\leq\alpha\leq 1$ and $0<\beta< 3+\alpha$. There exists a force $f$ for which there are two distinct solutions $\theta_1$ and $\theta_2$ to the $(\alpha,\beta)$-SQG equation \eqref{eq:SQG} with $\theta^\circ=0$. Moreover, 
\begin{equation}\label{mainthm:integrability}
\Lambda^s\theta_j\in L_t^pL^q,
\qquad
\Lambda^r f\in L_t^aL^b,
\end{equation}
for all $r,s\geq -1$ and $1\leq a,b,p,q\leq\infty$ in the regimes
\begin{equation}\label{mainthm:regime}
\frac{\beta}{p}+\frac{2}{q}>s+\beta-\alpha,
\qquad
\frac{\beta}{a}+\frac{2}{b}>r+2\beta-\alpha.
\end{equation}
Moreover, 
for $p=\infty$ the solutions are continuous in time
$$
\Lambda^s\theta_j \in C_t L^q
\quad\text{for all}\quad
\frac{2}{q}>s+\beta-\alpha,
$$
and belong to the critical space
$$
\Lambda^s\theta_j \in L_t^\infty L^{\frac{2}{s+\beta-\alpha}}
\quad\text{for all}\quad
s\leq \alpha-\beta+ 2.
$$
\end{thm}

\begin{Rem}\label{Rem:thm:main}
We provide several clarifications and refinements of Theorem \ref{thm:main}:
\begin{enumerate}[(i)]
    \item The solutions are global in time. However, the integrability in time degenerates as $t\to\infty$. Thus, equation \eqref{mainthm:integrability} must be understood as $L^p([0,T];L^q(\R^2))$ for all $T>0$. 
    \item The solutions are smooth for all $t>0$. Thus, they satisfy the $(\alpha,\beta)$-SQG equation \eqref{eq:SQG} in a classical sense. At time $t=0$, they satisfy the equation in the weak sense:
    $$
\int_0^\infty\int_{\R^2}
(\theta\partial_t\phi+\theta v\cdot\nabla\phi
-\theta\Lambda^\beta\phi
+\Lambda^{-1}f\Lambda\phi)\dif x\dif t
=-\int_{\R^2}\theta^\circ(x)\phi(0,x)\dif x,
$$
for all test function $\phi\in C_c^\infty([0,\infty)\times\R^2)$. In our case, the right hand side vanishes since $\theta^\circ=0$.
We recall that the weak formulation makes sense provided that $\theta,\theta v,\Lambda^{-1}f\in L_{t,x}^1$. 
First, Theorem~\ref{thm:main} directly yields $\theta_j,\,\Lambda^{-1}f \in L^1_{t,x}$.
Second, we claim that $\theta_j v_j \in L_{t,x}^1$. On the one hand,
$$
\theta_j\in L_t^p L^q
\quad\text{for}\quad
\frac{\beta}{p}+\frac{2}{q}>\beta-\alpha.
$$
On the other hand, since $v_j\sim\Lambda^{\alpha-1}\theta_j$, we have
$$
v_j\in L_t^{p'}L^{q'}
\quad\text{for}\quad
\frac{\beta}{p'}+\frac{2}{q'}>\beta-1.
$$
By imposing $\frac{1}{p'}=1-\frac{1}{p}$ and $\frac{1}{q'}=1-\frac{1}{q}$, we get the condition
$
\beta-\alpha<\frac{\beta}{p}+\frac{2}{q}<3,
$
which is possible provided that $\beta<3+\alpha$. 
    \item\label{Rem:thm:main:p=infty} 
    The statement for $p=\infty$ follows from the fact that $\theta^\circ=0$ together with the bound
    $$
    \|\Lambda^s\theta_j(t)\|_{L^q}
    \lesssim t^{\frac{1}{\beta}\left(\frac{2}{q}-s-\beta+\alpha\right)}.
    $$
    \item\label{thm:main:BMO} 
    The behavior in \ref{Rem:thm:main:p=infty} extends analogously to other functional settings. In general, one obtains $\theta_j \in C_t Y$ for any supercritical space $Y$, whereas $\theta_j \in L_t^\infty Y \setminus C_t Y$ for any critical space $Y$. For instance, the two distinct solutions to the forced 2D Navier–Stokes equation in Theorem~\ref{thm:LPS} satisfy
    $$
    v_j\in L_t^\infty BMO^{-1}.
    $$
    This does not contradict the global well-posedness for small data established by Koch and Tataru \cite{KochTataru01}, since continuity fails and the forcing is too singular at $t=0$.
    Remarkably, nonuniqueness in $BMO^{-1}$ has been recently established—for large data without forcing—on $\mathbb{T}^3$ by Coiculescu and Palasek \cite{CoiculescuPalasekpp}, and on $\mathbb{T}^2$ by Cheskidov, Dai, and Palasek \cite{CDPpp}.
\end{enumerate}
\end{Rem}


\subsection{Nonuniqueness of energy solutions}\label{sec:LH}

In this section, we examine the ranges of parameters for which the nonunique solutions from Theorem~\ref{thm:main} are \emph{$\dot H^s$-energy solutions}: 
$$
\theta\in C_t\dot{H}^s\cap L_t^2\dot{H}^{s+\frac{\beta}{2}}.
$$
That is, these solutions are continuous in $\dot H^s$, corresponding to $p = \infty$ and $q = 2$, with an additional gain in regularity due to diffusion.
The natural space for the forcing term associated with the $\dot H^s$ energy estimate is 
$$
f \in L_t^1\dot{H}^s + L_t^2\dot{H}^{s-\tiny\frac{\beta}{2}}.
$$

For $\alpha = 1$ and $0 < \beta < 2$, Miura \cite{Miura06} and Ju \cite{Ju06} established local well-posedness for the dissipative SQG equation in the critical case $s = 2 - \beta$:
$$
\theta\in C_t H^{2-\beta}\cap L_t^2\dot{H}^{2-\frac{\beta}{2}}.
$$  
This result was extended to global-in-time solutions in the critical case $\beta=1$ by Dong and Du \cite{DongDu08}. Earlier, Constantin, Córdoba, and Wu \cite{CCW00} proved global existence of solutions with $\theta \in L_t^\infty H^1$ and uniqueness in the class $\theta \in L_t^\infty H^2$ for the critical SQG equation with small $L^\infty$ initial data.

As a corollary of Theorem~\ref{thm:main}, we deduce nonuniqueness whenever $s + \beta - \alpha < 1$. In particular, for $\alpha = 1$ this shows the sharpness of the \emph{Miura–Ju class}, namely $s < 2 - \beta$.

\begin{thm}[Nonuniqueness of $\dot H^s$-energy solutions]\label{thm:energysolutions}
Let $0\leq\alpha\leq 1$ and $0<\beta<2+\alpha$.
There exists a force $f$ for which there are two distinct solutions $\theta_1$ and $\theta_2$ to the $(\alpha,\beta)$-SQG equation \eqref{eq:SQG} with $\theta^\circ=0$ such that
$$
\theta_j\in C_t\dot{H}^s\cap L_t^2\dot{H}^{s+\frac{\beta}{2}},
\qquad
f\in L_t^1\dot{H}^s\cap L_t^2\dot{H}^{s-\frac{\beta}{2}},
$$
for all $s+\beta-\alpha<1$.
\end{thm}
\begin{proof}
The nonunique solutions correspond to those constructed in Theorem~\ref{thm:main} with $q=b=2$.
First, taking $p=\infty$ and $a=1$, we get
$$
\Lambda^s\theta_j\in C_t L^2, \qquad \Lambda^s f\in L_t^1 L^2,
$$
provided that $\frac{\beta}{\infty}+\frac{2}{2}>s+\beta-\alpha$
and $\frac{\beta}{1}+\frac{2}{2}>s+2\beta-\alpha$
.
Second, taking $p=a=2$, we get
$$
\Lambda^{s+\frac{\beta}{2}}\theta_j\in L_t^2 L^2,\qquad
\Lambda^{s-\frac{\beta}{2}}f\in L_t^2 L^2,
$$
provided that $\frac{\beta}{2}+\frac{2}{2}>(s+\frac{\beta}{2})+\beta-\alpha$
and
$\frac{\beta}{2}+\frac{2}{2}>(s-\frac{\beta}{2})+2\beta-\alpha$.
All these conditions are equivalent to $s+\beta-\alpha<1$.
\end{proof}

Next, we examine the ranges of $\alpha$ and $\beta$ for which the nonunique solutions from Theorem~\ref{thm:energysolutions} are of Leray--Hopf type. By this we mean that they satisfy an appropriate form of an energy inequality. We distinguish two well-known energies:
\begin{equation}\label{eq:Hamiltonian}
\mathcal{H}(t):=
\frac{1}{2}\|\theta(t)\|_{\dot{H}^{\frac{\alpha-2}{2}}}^2
\quad\text{and}\quad
\mathcal{E}(t):=
\frac{1}{2}\|\theta(t)\|_{L^2}^2.
\end{equation}
These correspond, in the inviscid case, to the Hamiltonian and the $L^2$-Casimir, respectively, since they are conserved when $f = 0$. In the viscous case, a standard energy estimate shows that classical solutions to the $(\alpha,\beta)$–SQG equation \eqref{eq:SQG} satisfy, for both $s = \frac{\alpha-2}{2}$ and $s = 0$, the energy identity
\begin{equation}\label{eq:Leray:SQG}
\frac{1}{2}\|\theta(t)\|_{\dot{H}^s}^2
+\int_0^t\|\theta\|_{\dot{H}^{s+\frac{\beta}{2}}}^2\dif t'
=\frac{1}{2}\|\theta^\circ\|_{\dot{H}^s}^2
+\int_0^t\langle f,\theta\rangle_{\dot{H}^s}\dif t'.
\end{equation}
For weak solutions obtained as limits of suitably regularizing mechanisms, the energy identity holds in the form of an inequality.


\subsubsection{$\dot{H}^{\frac{\alpha-2}{2}}$--energy solutions}

We begin by discussing the more familiar case of the 2D Navier--Stokes equation. For $\alpha=0$, the vorticity $\omega:=\nabla^\perp\cdot v$ plays the role of the temperature $\theta$, and the (divergence-free) velocity field $v$ satisfies the (fractional) Navier--Stokes equation
\begin{equation}\label{eq:NS:v:beta}
\partial_t v + v\cdot\nabla v
+\Lambda^\beta v= -\nabla p + g,
\end{equation}
where $p$ is the pressure and $g$ an external force. 
Applying the curl to the momentum equation \eqref{eq:NS:v:beta} yields the vorticity formulation
\begin{equation}\label{eq:NS:w}
\partial_t\omega + v\cdot\nabla\omega+\Lambda^\beta\omega=f,
\end{equation}
where $f=\nabla^\perp\cdot g$. 
Note that \eqref{eq:NS:w} is simply \eqref{eq:SQG:1} with $\theta$ replaced by $\omega$.
Recall that $v$ and $g$ are recovered from $\omega$ and $f$, respectively, through the Biot–Savart law: $v=\nabla^\perp\Delta^{-1}\omega$ and $g=\nabla^\perp\Delta^{-1}f$. 

In the celebrated work \cite{Leray1934}, Leray proved the global existence of solutions to the standard Navier–Stokes equation ($\beta=2$)
$$
v\in C_tL^2\cap L_t^2\dot{H}^1,
$$
by constructing a regularizing sequence for \eqref{eq:NS:v:beta} and then passing to the limit using the compactness provided by the energy inequality \begin{equation}\label{eq:Leray:NS}
\frac{1}{2}
\|v(t)\|_{L^2}^2
+\int_0^t\|\nabla v\|_{L^2}^2\dif t'
\leq \frac{1}{2}
\|v^\circ\|_{L^2}^2
+\int_0^t\langle g,v\rangle_{L^2}\dif t'.
\end{equation}
These solutions are usually referred to as \emph{Leray--Hopf solutions}, also recognizing Hopf’s contribution in the setting of bounded domains \cite{Hopf51}.
Observe that \eqref{eq:Leray:NS}
corresponds to \eqref{eq:Leray:SQG} with an inequality for $\alpha=0$ and $s=-1$, since the Hamiltonian can be written as
$$
\mathcal{H}=\frac{1}{2}\|\omega\|_{\dot{H}^{-1}}^2=\frac{1}{2}\|v\|_{L^2}^2.
$$
Moreover, for $\beta = 2$ the quantity
$$
\mathcal{E}=\frac{1}{2}\|\omega\|_{L^2}^2
=\frac{1}{2}\|\nabla v\|_{L^2}^2
$$
corresponds to the enstrophy, which satisfies \eqref{eq:Leray:SQG} with an inequality in two dimensions, namely
\begin{equation}\label{eq:Leray:NS:omega}
\frac{1}{2}
\|\omega(t)\|_{L^2}^2
+\int_0^t\|\nabla\omega\|_{L^2}^2\dif t'
\leq \frac{1}{2}
\|\omega^\circ\|_{L^2}^2
+\int_0^t\langle f,\omega\rangle_{L^2}\dif t'.
\end{equation}
This control allows one to conclude that the Navier–Stokes equation is globally well posed in the Leray–Hopf class for $\beta = 2$. However, as mentioned previously, uniqueness of Leray–Hopf solutions 
\begin{equation}\label{eq:Leray:beta}
v\in C_tL^2\cap L_t^2\dot{H}^{\tiny\frac{\beta}{2}},
\end{equation}
no longer holds in the hypodissipative case $\beta < 2$, at least in the presence of forcing \cite{AC}.

For the dissipative SQG equation, for which
$\mathcal{H} = \tfrac{1}{2}\|\theta\|_{\dot H^{-\frac{1}{2}}}^2$,
Marchand proved in \cite{Marchandexistence} a natural extension of Leray’s existence theorem in the class
\begin{equation}\label{eq:Marchand}
\theta\in C_t\dot{H}^{-\frac{1}{2}}
\cap L_t^2\dot{H}^{\frac{\beta-1}{2}},
\end{equation}
which we refer to as \emph{Marchand solutions}. 

As a corollary of Theorem~\ref{thm:energysolutions}, we deduce nonuniqueness of $\dot{H}^{\frac{\alpha-2}{2}}$-energy solutions for $\beta < 2 + \frac{\alpha}{2}$. On the one hand, this recovers the result of Albritton and Colombo on nonuniqueness of Leray–Hopf solutions \eqref{eq:Leray:beta} for the hypodissipative 2D Navier–Stokes equation \cite{AC}. On the other hand, it shows that Marchand solutions \eqref{eq:Marchand} are not necessarily unique for $\beta < \frac{5}{2}$. This exponent coincides with the uniqueness threshold introduced by Lions for the 3D Navier–Stokes equation \cite{Lions69}, as evidenced by the nonuniqueness results of Luo and Titi \cite{LuoTiti20}, and of Khor, Miao, and Su \cite{KMS23}.

\begin{thm}[Nonuniqueness of Leray--Hopf and Marchand solutions]\label{thm:Leray:H}
Let
$$
0\leq\alpha\leq 1,
\qquad
0<\beta<2+\frac{\alpha}{2}.
$$
There exists a force $f$ for which there are two distinct solutions $\theta_1$ and $\theta_2$ to the $(\alpha,\beta)$-SQG equation \eqref{eq:SQG} with $\theta^\circ=0$ such that
\begin{equation}\label{eq:LerayMarchand}
\theta_j\in C_t\dot{H}^{\frac{\alpha-2}{2}}\cap L_t^2\dot{H}^{\frac{\alpha+\beta-2}{2}},\qquad
f\in L_t^1\dot{H}^{\frac{\alpha-2}{2}}\cap L_t^2\dot{H}^{\frac{\alpha-\beta-2}{2}}.
\end{equation}
Moreover, they satisfy the energy identity \eqref{eq:Leray:SQG} with $s=\frac{\alpha-2}{2}$.
\end{thm}
\begin{proof}
The nonunique solutions correspond to those constructed in Theorem~\ref{thm:energysolutions} with $s = \frac{\alpha-2}{2}$. Since these solutions are classical for $t > 0$, they satisfy the energy identity \eqref{eq:Leray:SQG}
\begin{equation}\label{eq:Leray:H}
\frac{1}{2}\|\theta_j(t)\|_{\dot{H}^{\frac{\alpha-2}{2}}}^2
+\int_{t_0}^t\|\theta_j\|_{\dot{H}^{\frac{\alpha+\beta-2}{2}}}^2\dif t'
=\frac{1}{2}\|\theta_j(t_0)\|_{\dot{H}^{\frac{\alpha-2}{2}}}^2
+\int_{t_0}^t\langle f,\theta_j\rangle_{\dot{H}^{\frac{\alpha-2}{2}}}\dif t',
\end{equation}
for all $0<t_0<t$. Letting $t_0 \to 0$ and using \eqref{eq:LerayMarchand}, we obtain the same identity for $t_0=0$.
\end{proof}

Nonunique Leray–Hopf solutions to the unforced hypodissipative Navier–Stokes equation on $\mathbb{T}^3$ were constructed by Colombo, De Lellis, and De Rosa \cite{ColomboDLellisDeRosa} for $0<\beta<\tfrac{2}{5}$, later improved to $0<\beta<\tfrac{2}{3}$ by De Rosa in \cite{DeRosa}.
Nonuniqueness for the dissipative SQG equation on $\mathbb{T}^2$ was also established—under suitable Hölder regularity of $\Lambda^{-1}\theta$—by Buckmaster, Shkoller, and Vicol \cite{BSV19}, and in the forced case by Dai and Peng \cite{DaiPengpp}.
Earlier, Buckmaster and Vicol \cite{BuckVicol} constructed the first nonunique solutions
$v \in C_t L^2 \cap L_t^{2} H^\gamma$—for some small $\gamma>0$—to the Navier–Stokes equation on $\mathbb{T}^3$ (see also \cite{LuoQu20,LuoTiti20,BMS21,BCV22,MNY24,KinraKoley24,HZZ25,RomitoTriggiano25,CZZpp}).
These works rely on convex integration, introduced in fluid mechanics by De Lellis and Székelyhidi for the Euler equation \cite{DeLellisSzekelyhidi09}. More recently, Palasek and Coiculescu combined this method with dyadic models to establish nonuniqueness for critical Navier–Stokes data \cite{CoiculescuPalasekpp} (see also \cite{CDPpp}).

As mentioned in the introduction, nonuniqueness of Leray–Hopf solutions to the Navier–Stokes equation in $\mathbb{R}^3$ has been established by Albritton, Bru\`e, and Colombo in the forced case \cite{ABC22}, and very recently by Hou, Wang, and Yang in the unforced case \cite{HWYpp}.


\subsubsection{$L^2$--energy solutions}

Prior to Marchand’s existence theorem, Resnick had already applied a Leray-type argument—based on the energy $\mathcal{E}$ rather than $\mathcal{H}$—to deduce the existence of solutions to the dissipative SQG equation
\begin{equation}\label{eq:Resnick}
\theta\in C_tL^2\cap L_t^2\dot{H}^{\frac{\beta}{2}},
\end{equation}
which we refer to as \emph{Resnick solutions}.
In the subcritical regime $\beta > 1$, these solutions become instantly smooth, as shown by Constantin and Wu \cite{ConstantinWu99}. In the critical case $\beta = 1$, Caffarelli and Vasseur proved that such solutions are at least Hölder continuous for positive times \cite{CaffarelliVasseur10}. It is worth emphasizing that these results do not imply uniqueness. For smooth initial data, Kiselev, Nazarov, and Volberg \cite{KZN07} proved the existence of a unique global smooth solution $\theta\in L_t^\infty W^{1,\infty}$.

As a corollary of Theorem~\ref{thm:energysolutions}, we deduce nonuniqueness of $L^2$-energy solutions in the regime $\beta < 1 + \alpha$. For $\alpha = 1$, this shows that Resnick solutions \eqref{eq:Resnick} are not necessarily unique for $\beta < 2$.

\begin{thm}[Nonuniqueness of Resnick solutions]\label{thm:Leray:E}
Let
$$
0\leq\alpha\leq 1,
\qquad
0<\beta<1+\alpha.
$$
There exists a force $f$ for which there are two distinct solutions $\theta_1$ and $\theta_2$ to the $(\alpha,\beta)$-SQG equation \eqref{eq:SQG} with $\theta^\circ=0$ such that
$$
\theta_j\in C_t L^2\cap L_t^2\dot{H}^{\frac{\beta}{2}},\qquad
f\in L_t^1 L^2 \cap L_t^2\dot{H}^{-\tiny\frac{\beta}{2}}.
$$
Moreover, they satisfy the energy identity \eqref{eq:Leray:SQG} with $s=0$.
\end{thm}
\begin{proof}
It follows analogously to the proof of Theorem~\ref{thm:Leray:H} by taking $s = 0$ instead of $s = \frac{\alpha-2}{2}$.
\end{proof}

\subsection{Nonuniqueness in $L_t^pL^q$ spaces}\label{sec:LPS}

In the previous section, we saw that uniqueness need not hold within the natural class of energy solutions. For the Navier–Stokes equation, the classical works of Prodi \cite{Prodi}, Serrin \cite{Serrin}, and Ladyzhenskaya \cite{Ladyzhenskaya} showed that uniqueness can be recovered by additionally assuming that the solutions belong to suitable $L_t^p L^q$ spaces. Similar criteria were later established for the generalized SQG equation.

In this section, we explain how Theorem~\ref{thm:main} yields sharp nonuniqueness results for specific values of $s$. We observe that the complement of the regimes appearing in Theorem~\ref{thm:main} can be interpreted as a generalized Ladyzhenskaya–Prodi–Serrin condition. While such conditions are known to guarantee uniqueness for certain parameter ranges (see the references in this section), it remains an interesting open question whether uniqueness holds throughout the full range.

We begin by discussing the more familiar case of the standard 2D Navier--Stokes equation, and then turn to the generalized SQG equation.

\subsubsection{The Navier-Stokes case}
For $\alpha=0$ and $\beta=2$, the Navier--Stokes equation \eqref{eq:NS:v:beta} reads as
\begin{equation}\label{eq:NS:v}
\partial_t v + v\cdot\nabla v
= -\nabla p + \Delta v + g.
\end{equation}

The Ladyzhenskaya–Prodi–Serrin (LPS) criterion asserts that if two solutions $v_1$ and $v_2$ of the Navier-Stokes equation \eqref{eq:NS:v} in dimension $d\geq 2$ satisfy
$$
v_j\in L_t^p L^q
\quad\text{with}\quad
\frac{2}{p}+\frac{d}{q}\leq 1,
$$
for certain ranges of $q$ to be discussed below, then necessarily $v_1=v_2$. 

This uniqueness criterion was first proved in \cite{Prodi,Serrin} under the stronger assumption that the $v_j$ are Leray solutions (later shown to be smooth in \cite{Ladyzhenskaya}), that is, they satisfy the energy inequality \eqref{eq:Leray:NS}
in the range $d<q<\infty$. The extension to the critical endpoint $q=d$ and $p=\infty$ was obtained through the uniqueness result of Kozono and Sohr \cite{KozonoSohr}, whose corresponding smoothness was established by Escauriaza, Serëgin, and Šverák in \cite{EscauriazaSereginShverak}.

Without assuming the Leray condition \eqref{eq:Leray:NS}, the same uniqueness criterion for $d<q<\infty$ was proved by Fabes, Jones, and Rivière \cite{FJR72}. The critical endpoint $q=d$ and $p=\infty$ turned out to be more delicate: one must additionally impose time continuity, namely 
$$
v\in C_tL^d,
$$
for $d=2,3$, rather than only $v\in L_t^\infty L^2$ (see e.g.~\cite{FLRT00,Meyer97,Mon99,LionsMasmoudi01}). Interestingly, for dimensions $d\geq 4$ this time-continuity assumption is no longer needed \cite{LionsMasmoudi01}.

As a corollary of Theorem \ref{thm:main}, we deduce nonuniqueness below the \emph{Ladyzhenskaya–Prodi–Serrin class} for the forced 2D Navier-Stokes equation. The extension to the generalized SQG equation follows simply by taking $s = -1$ in Theorem~\ref{thm:main}.

\begin{thm}[Nonuniqueness below the Ladyzhenskaya–Prodi–Serrin class]\label{thm:LPS}
There exists a force 
$$
g\in L_t^1 L^b
\quad\text{for all}\quad
1\leq b<2,
$$
for which there are two distinct solutions $v_1$ and $v_2$ to the 2D Navier-Stokes equation \eqref{eq:NS:v} with initial datum $v^\circ=0$ such that
$$v_j\in L_t^pL^q
\quad\text{for all}\quad
\frac{2}{p}+\frac{2}{q}>1.$$
Moreover, for $p=\infty$ it holds
$$
v_j\in C_t L^q\cap L_t^\infty L^2
\quad\text{for all}\quad
1\leq q<2.
$$
\end{thm}
\begin{proof}
For $\alpha=0$, $\beta=2$, we take $r=s=-1$ and $a=1$ in Theorem \ref{thm:main}. 
Recall that, by the Biot-Savart law, we have 
$v_j\sim\Lambda^{-1}\omega_j$ and $g\sim\Lambda^{-1}f$.
\end{proof}

Observe that the solutions in Theorem~\ref{thm:LPS} cannot be Leray–Hopf solutions.
On the one hand, we have $g \notin L_t^1 L^2$ and $v_j \notin C_t L^2$.
On the other hand, taking $p=q=2$ and $\gamma = 1+s$, we deduce that
$v_j \in L_t^2 \dot{H}^\gamma$ for all $\gamma < 1$, while $v_j \notin L_t^2 \dot{H}^1$. Similarly, $g\notin L_t^2\dot{H}^{-1}$.

By adjusting the parameters, the solutions and the forcing term can be described in other functional spaces.
For instance, taking $s=0$ and $p=1$, we obtain $\omega_j \in L_t^1 L^q$ for all $q < \infty$, so that they lie just below the Beale--Kato--Majda class.\\

The nonuniqueness below the LPS class for the unforced Navier–Stokes equation in $\T^d$ was first proved by Cheskidov and Luo in the endpoint cases: $L^p_tL^\infty$ with $p<2$ and $d\geq 2$ in \cite{CheskidovLuo22}, and in $C_tL^q$ with $q<2$ and $d=2$ in \cite{CheskidovLuo23}
(see also \cite{MNY24sharp,MiaoZhaopp}). 
The case $L_t^2L^q$ with $q<\infty$ and $d\geq 2$ was recently established in \cite{CDPpp}.

To the best of our knowledge, sharp nonuniqueness below the LPS class remains open in the unforced 2D setting for $p\neq 2,\infty$ and $q\neq\infty$. Notably, the nonunique Leray-Hopf solutions in $\mathbb{R}^3$ recently constructed in \cite{HWYpp} lie just below the LPS class.

\subsubsection{The generalized SQG case}
In subsequent works, similar uniqueness criteria have been established for other values of $\alpha$. For the dissipative SQG equation ($\alpha = 1$), Constantin and Wu \cite{ConstantinWu99} proved uniqueness of Resnick solutions $\theta \in L_t^\infty L^2 \cap L_t^2 H^{\tiny\frac{\beta}{2}}$ with forcing term $f \in L_t^2 H^{-\tiny\frac{\beta}{2}}$ in the subcritical regime $1 < \beta \le 2$, when the solution additionally satisfies the  LPS-type condition
$$
\theta\in L_t^pL^q
\quad\text{with}\quad
\frac{\beta}{p}+\frac{2}{q}=\beta-1,
\quad q\geq 1.
$$ 

The following corollary of Theorem~\ref{thm:main} shows the sharpness of the \emph{Constantin–Wu class} for the dissipative SQG equation. Notice that the regime considered in \cite{ConstantinWu99} is $0 < \beta \le 2$, whereas in Theorem~\ref{thm:ConstantinWu} we allow $0 < \beta < 4$. In addition, for $0 < \beta < 2$ the solutions constructed here belong to the Resnick class.

We remark that the theorem is stated for $\alpha = 1$ in order to facilitate comparison with \cite{ConstantinWu99}. The extension to the generalized SQG equation follows by taking $s = 0$ in Theorem~\ref{thm:main}. For $\alpha = 0$, this can be interpreted as controlling $\nabla v \in L_t^p L^q$, which may be of interest in 3D in view of Beirão da Veiga’s uniqueness criterion \cite{daVeiga95}.

\begin{thm}[Nonuniqueness below the Constantin–Wu class]\label{thm:ConstantinWu}
Let $\alpha=1$ and $0<\beta<4$.
There exists a force $f$ for which there are two distinct solutions $\theta_1$ and $\theta_2$ to the $(1,\beta)$-SQG equation \eqref{eq:SQG} with $\theta^\circ=0$ such that
$$
\theta_j\in L_t^p L^q
\quad\text{for all}\quad
\frac{\beta}{p}+\frac{2}{q}>\beta-1.
$$
Moreover, for $p=\infty$ it holds
$$
\theta_j\in C_t L^q\cap L_t^\infty L^{\frac{2}{\beta-1}}
\quad\text{for all}\quad
\frac{2}{q}>\beta-1.
$$
Furthermore, if $\beta < 2$, then $\theta_j\in L_t^\infty L^2\cap L_t^2\dot{H}^{\tiny\frac{\beta}{2}}$
and $f\in L_t^1L^2\cap L_t^2\dot{H}^{-\tiny\frac{\beta}{2}}$.
\end{thm}
\begin{proof}
For $\alpha=1$, we take $s=0$ in Theorem \ref{thm:main}. For $\beta<2$, Theorem \ref{thm:Leray:E} ensures that they are $L^2$-energy solutions.
\end{proof}

Later, in the regime $0 < \beta \le 2$, Dong and Chen \cite{DongChen2006,DongChen12} established another uniqueness criterion by controlling the gradient, namely under the following LPS-type condition
$$
\nabla\theta\in L_t^pL^q
\quad\text{with}\quad
\frac{\beta}{p}+\frac{2}{q}=\beta,
\quad
\frac{2}{\beta}<q<\infty.
$$
Regularity of solutions in this class had previously been established by Chae \cite{Chae06}. 

This uniqueness criterion was later extended by Zhao and Liu \cite{ZhaoLiu13} to the generalized SQG equation ($0 < \alpha \le 1$) in the regime $0 < \beta < 2\alpha$. In this case, they assumed
$$
\nabla\theta\in L_t^pL^q
\quad\text{with}\quad
\frac{\beta}{p}+\frac{2}{q}=1+\beta-\alpha,
\quad
\frac{2}{1+\beta-\alpha}<q<\infty.
$$

The following corollary of Theorem~\ref{thm:main} shows the sharpness of the \emph{Dong–Chen–Zhao–Liu class}. Notice that the regime considered in \cite{DongChen2006,DongChen12,ZhaoLiu13} is $0 < \beta < 2\alpha$, whereas in Theorem~\ref{thm:ConstantinWu} we allow $0 < \beta < 3 + \alpha$. In addition, for $0 < \beta < 1 + \alpha$ our solutions belong to the Resnick class.

\begin{thm}[Nonuniqueness below the Dong–Chen–Zhao–Liu class]\label{thm:DCZL}
Let $0\leq\alpha\leq 1$ and $0<\beta<3+\alpha$.
There exists a force $f$ for which there are two distinct solutions $\theta_1$ and $\theta_2$ to the $(\alpha,\beta)$-SQG equation \eqref{eq:SQG} with $\theta^\circ=0$ such that
$$
\nabla\theta_j\in L_t^p L^q
\quad\text{for all}\quad
\frac{\beta}{p}+\frac{2}{q}>1+\beta-\alpha.
$$
Moreover, for $p=\infty$ it holds
$$
\nabla\theta_j\in C_t L^q\cap L_t^\infty L^{\frac{2}{1+\beta-\alpha}}
\quad\text{for all}\quad
\frac{2}{q}>1+\beta-\alpha.
$$
Furthermore, if $\beta < 1+\alpha$, then $\theta_j\in L_t^\infty L^2\cap L_t^2\dot{H}^{\tiny\frac{\beta}{2}}$
and $f\in L_t^1 L^2\cap L_t^2\dot{H}^{-\tiny\frac{\beta}{2}}$.
\end{thm}
\begin{proof}
For $\alpha=1$, we take $s=1$ in Theorem \ref{thm:main}. For $\beta<1+\alpha$, Theorem \ref{thm:Leray:E} ensures that they are $L^2$-energy solutions.
\end{proof}

\subsubsection{Other uniqueness criteria}
We remark that in \cite{DongChen12}, Dong and Chen proved uniqueness of $L^2$-energy solutions for the dissipative SQG equation in a broader Besov class, namely 
$$
\nabla\theta\in L_t^p B_{q,\infty}^0
\quad\text{with}\quad
\frac{\beta}{p}+\frac{2}{q}=\beta,
\quad
\frac{2}{\beta}<q<\infty.
$$
Earlier, Abidi and Hmidi \cite{AbidiHmidi08} proved the existence of a unique global solution for critical SQG equation in the Besov class
$\theta\in C_t\dot{B}^0_{\infty,1}\cap L_t^1\dot{B}^1_{\infty,1}.$
See e.g.~\cite{BCDbook,ChaeLee03} for further uniqueness criteria in Besov settings. 

Although our nonunique solutions are expected to fall just outside these classes as well, we have chosen to state our results in terms of Lebesgue and Sobolev spaces in order to avoid introducing additional parameters and notation. The interested reader may compute the scaling of the solutions from Theorem~\ref{thm:main}, as in Section~\ref{sec:selfsimilarcoordinates}, and verify the corresponding Besov regimes.

We conclude by mentioning that several other uniqueness criteria have been introduced for the generalized SQG equation. We recall the classical well-posedness result for small data by Koch and Tataru \cite{KochTataru01} for the Navier–Stokes equation in the critical class $\it{BMO}^{-1}$. As mentioned in Remark~\ref{Rem:thm:main}\ref{thm:main:BMO}, the two distinct solutions to the forced 2D Navier–Stokes equation constructed in Theorem~\ref{thm:LPS} satisfy $v_j\in L_t^\infty\it{BMO}^{-1}$. See also the recent works of Coiculescu and Palasek \cite{CoiculescuPalasekpp}, and of Cheskidov, Dai, and Palasek \cite{CDPpp}.

In this regard, Marchand showed in \cite{Marchand08} that solutions $\theta\in L_t^2 L^2\cap L_t^\infty \dot{H}^{-\frac{1}{2}}$ to the dissipative SQG equation are unique provided that they are small in $L_t^\infty \it{BMO}$. Consequently, this smallness condition must fail for the nonunique Marchand solutions constructed in Theorem~\ref{thm:Leray:H}.

Moreover, Liu, Jia, and Dong \cite{LiuJiaDong2012} proved that for $\alpha = 1$ and $0 < \beta < 2$, solutions with initial data $\theta^\circ \in H^{2+\beta}$ satisfying
$\theta\in L_t^\infty L^2\cap L_t^2 H^\frac{\beta}{2}$
are unique provided that, in addition, $\nabla \theta \in L_t^1 \it{BMO}$. As a consequence, the solutions constructed in Theorem~\ref{thm:Leray:H} cannot belong to this class.


\section{Self-similar coordinates}\label{sec:selfsimilarcoordinates}

We consider the self-similar variables
\begin{equation}\label{eq:sscoordinates}
X=\frac{x}{t^{\nicefrac{1}{\beta}}},
\qquad
\tau=\frac{1}{\nu}\log t,
\end{equation}
in terms of a parameter $\nu>0$, to be determined. The choice of the letter $\nu$ is due to the fact that it will appear later in front of the fractional Laplacian, thus representing a sort of viscosity.

\begin{prop}\label{prop:selfsimilarchangevariables}
The pair $(\theta,f)$ given by the change of variables
\begin{subequations}\label{eq:oOfF}
\begin{align}
\theta(t,x)
&=\frac{t^{\frac{\alpha}{\beta}-1}}{\nu}\Theta(\tau,X),\\
f(t,x)
&=\frac{t^{\frac{\alpha}{\beta}-2}}{\nu^2}F(\tau,X),
\end{align}
\end{subequations}
is a solution to the $(\alpha,\beta)$-SQG equation \eqref{eq:SQG} if and only if the pair $(\Theta,F)$ solves the \emph{self-similar $(\alpha,\beta)$-SQG equation}
\begin{equation}\label{eq:abSQG_selfsimilar}
\partial_\tau\Theta
+V\cdot\nabla\Theta+\nu J\Theta=F,
\end{equation}
where $J=J_{\alpha,\beta}$ is given by
$$
J\Theta
=\Lambda^\beta\Theta
+\left(\frac{\alpha}{\beta}-1\right)\Theta
-\frac{1}{\beta}X\cdot\nabla\Theta.
$$
The corresponding velocities are linked by
\begin{equation}\label{eq:vV}
v(t,x)=\frac{t^{\frac{1}{\beta}-1}}{\nu}V(\tau,X),
\end{equation}
that is, $(v,V)$ are recovered from $(\theta,\Theta)$ through the $\alpha$-Biot Savart law \eqref{eq:SQG:BS}, respectively.
\end{prop}
\begin{proof}
A straightforward computation shows that 
$$
\partial_t\theta
=\frac{t^{\frac{\alpha}{\beta}-2}}{\nu^2}
\left(\partial_\tau\Theta
+\nu\left(\frac{\alpha}{\beta}-1\right)\Theta
-\frac{\nu}{\beta}X\cdot\nabla\Theta\right).
$$
Moreover, using Lemma~\ref{lemma:tellme}, we also obtain \eqref{eq:vV}, as well as
$$v\cdot\nabla\theta
=\frac{t^{\frac{\alpha}{\beta}-2}}{\nu^2} V\cdot\nabla\Theta,$$
and
$$
\Lambda^\beta\theta
=\frac{t^{\frac{\alpha}{\beta}-2}}{\nu} 
\Lambda^\beta\Theta.
$$
This concludes the proof.
\end{proof}

\begin{prop}\label{prop:scalingtheta}
It holds
$$
\|\Lambda^s\theta\|_{L_t^p L^q}
=\frac{1}{\nu}
\left(\int_0^T t^{p\left(\frac{\alpha-s}{\beta}+\frac{2}{\beta q}-1\right)}
\|\Lambda^s\Theta(\tau)\|_{L^q}^p\dif t\right)^\frac{1}{p},
$$
where recall $\tau=\frac{1}{\nu}\log t$.
\end{prop}
\begin{proof}
Using Lemma~\ref{lemma:tellme}, we compute
$$
\Lambda^s\theta
=\frac{t^{\frac{\alpha-s}{\beta}-1}}{\nu}\Lambda^s\Theta.
$$
Hence, we have
$$
\int_{\R^2}|\Lambda^s\theta|^q\dif x
=\frac{t^{q\left(\frac{\alpha-s}{\beta}-1\right)+\frac{2}{\beta}}}{\nu^q}
\int_{\R^2}|\Lambda^s\Theta|^q\dif X.
$$
This concludes the proof.
\end{proof}

Analogously, we obtain the scaling for the force.

\begin{prop}\label{prop:scalingforce}
It holds
$$
\|\Lambda^rf\|_{L_t^a L^b}
=\frac{1}{\nu^2}
\left(\int_0^T t^{a\left(\frac{\alpha-r}{\beta}+\frac{2}{\beta b}-2\right)}
\|\Lambda^sF(\tau)\|_{L^q}^a\dif t\right)^\frac{1}{a},
$$
where recall $\tau=\frac{1}{\nu}\log t$.
\end{prop}

Note that the integrals in Propositions \ref{prop:scalingtheta} and \ref{prop:scalingforce} are finite provided that
$$p\left(\frac{\alpha-s}{\beta}+\frac{2}{\beta q}-1\right)>-1, \qquad a\left(\frac{\alpha-r}{\beta}+\frac{2}{\beta b}-2\right)>-1,$$
which agree with the regimes \eqref{mainthm:regime}.

\section{Proof of Theorem \ref{thm:main}}

We aim to construct a family of distinct solutions $\theta_\epsilon$, in terms of a parameter $\epsilon\geq 0$, to the $(\alpha,\beta)$-SQG equation \eqref{eq:SQG} for some external forcing term $f$.
According to Proposition \ref{prop:selfsimilarchangevariables}, we can express these solutions in self-similar coordinates \eqref{eq:sscoordinates} as 
\begin{subequations}\label{eq:oOfF:epsilon}
\begin{align}
\theta_\epsilon(t,x)
&:=\frac{t^{\frac{\alpha}{\beta}-1}}{\nu}\Theta_\epsilon(\tau,X),\\
f(t,x)
&:=\frac{t^{\frac{\alpha}{\beta}-2}}{\nu^2}F(\tau,X),
\end{align}
\end{subequations}
for some solutions $\Theta_\epsilon$ to the self-similar $(\alpha,\beta)$-SQG equation \eqref{eq:abSQG_selfsimilar} with an external forcing $F$, to be determined.
We split these solutions as
\begin{equation}\label{eq:Thetaeps}
\Theta_\epsilon:=\bar{\Theta}+\epsilon\tilde{\Theta},
\end{equation}
for some temperature $\bar{\Theta}$ independent of $\tau$, and a deviation $\tilde{\Theta}$, which we require to satisfy
$$
\tilde{\Theta}|_{\tau=-\infty}=0.
$$
The self-similar $(\alpha,\beta)$-SQG equation \eqref{eq:abSQG_selfsimilar} is written in terms of this decomposition \eqref{eq:Thetaeps} as
\begin{equation}\label{eq:SQG:deviation}
\bar{V}\cdot\nabla\bar{\Theta}
+\nu J\bar{\Theta}
+\epsilon(\partial_\tau-L_\nu)\tilde{\Theta}
+\epsilon^2\tilde{V}\cdot\nabla\tilde{\Theta}=F,
\end{equation}
where the velocities $\bar{V}$ and $\tilde{V}$ are recovered from $\bar{\Theta}$ and $\tilde{\Theta}$, respectively, through the $\alpha$-Biot-Savart law \eqref{eq:SQG:BS}, and $L_\nu=L_{\alpha,\beta,\nu,\bar{\Theta}}$ is the linearization of the self-similar $(\alpha,\beta)$-SQG equation \eqref{eq:abSQG_selfsimilar} around the steady temperature $\bar{\Theta}$:
\begin{equation}\label{eq:Lnu}
L_\nu\Theta
:=-\bar{V}\cdot\nabla\Theta
-V\cdot\nabla\bar{\Theta}
-\nu J\Theta.
\end{equation}
Notice that $L_0$ formally corresponds to the linearization of the $\alpha$-SQG equation \eqref{eq:SQG:nodif} (without dissipation) in the original system of coordinates. 

Our goal is to prove that the forcing $F$ above can be chosen to be independent of $\epsilon$.
We start by focusing on the first-order term in $\epsilon$, corresponding to the linear evolution equation
$$
(\partial_\tau-L_\nu)\Theta^{\text{lin}}=0.
$$
By separation of variables, we deduce that 
$$
\Theta^{\text{lin}}
(\tau,X)
=\Re(e^{\lambda\tau}W(X)),
$$
for some eigenpair $(\lambda,W)$ satisfying
\begin{equation}\label{eq:eigenvalueproblem}
L_\nu W=\lambda W.
\end{equation}

In Section \ref{sec:selfsimilarinstability} we prove the existence of a \emph{self-similarly unstable vortex} $\bar{\Theta}$, meaning that $L_{\nu,\bar{\Theta}}$ admits an eigenpair with $\Re\lambda>0$. 
This guarantees that 
$$
\Theta^{\text{lin}}|_{\tau=-\infty}=0.
$$
To this end, we apply Vishik's spectral argument \cite{Vishikpp1,Vishikpp2} to the unstable vortices constructed in \cite{CFMSSQG}. The main novelty is that we incorporate the diffusion as a perturbation. 

\begin{thm}[Self-similar instability]\label{thm:selfsimilarinstability}
Let $0\leq\alpha\leq 1$ and $0<\beta< 3+\alpha$.
There exists a vortex $\bar{\Theta}\in C_c^\infty$ with zero-mean satisfying that, for some $\nu>0$, there exists $\lambda_\nu\in\C$ with $\Re\lambda_\nu>0$ and $W_\nu\in H^k(\R^2)$ for all $k\in\N$, solving the eigenvalue problem \eqref{eq:eigenvalueproblem} for $L_\nu=L_{\alpha,\beta,\nu,\bar{\Theta}}$ given in \eqref{eq:Lnu}.
\end{thm}

In a recent preprint \cite{DolceMescolinipp}, Dolce and Mescolini revived a clever trick from Golovkin \cite{Golovkin64} which allows proving nonuniqueness once self-similar instability is established at the linear level. 
Golovkin's trick simply consists of taking the deviation as
$$
\tilde{\Theta}=\Theta^\text{lin},
$$
which requires taking the following force
$$
F=
\bar{V}\cdot\nabla\bar{\Theta}
+\nu J\bar{\Theta}
+V^{\text{lin}}\cdot\nabla\Theta^{\text{lin}}.
$$
It is immediate that both ($\epsilon=\pm 1$)
$$
\Theta_+=\bar{\Theta}+\Theta^\text{lin},
\qquad
\Theta_-
=\bar{\Theta}-\Theta^\text{lin},
$$
solve the equation \eqref{eq:SQG:deviation} for  this $F$.
The smoothness of the solutions allows us to conclude Theorem \ref{thm:main} through the Sobolev scaling (Propositions \ref{prop:scalingtheta} and \ref{prop:scalingforce}). 

\begin{Rem}
Alternatively, one can consider Vishik's forcing
$$
F=\bar{V}\cdot\nabla\bar{\Theta}
+\nu J\bar{\Theta}.
$$
This requires decomposing the deviation as
$$
\tilde{\Theta}
=\Theta^{\text{lin}}
+\epsilon\Theta^{\text{cor}},
$$
where the correcting term must satisfy
$$
(\partial_\tau-L_\nu)
\Theta^{\text{cor}}
=-\tilde{V}\cdot\nabla\tilde{\Theta}.
$$
To ensure that the solutions are different one needs to verify that such a $\Theta^{\text{cor}}$ exists satisfying the asymptotics
$$
\Theta^{\text{cor}}
=o(e^{\lambda\tau}),
\qquad\tau \to -\infty.
$$
This argument is more involved, but would allow to construct not only two, but infinitely many different solutions starting from the same initial datum and with the same radial forcing. The last step to complete the Jia–Šverák program would be to find a self-similarly unstable $\bar{\Theta}$ for which $F = 0$. 
\end{Rem}

\section{Self-similar instability}\label{sec:selfsimilarinstability}

In this section we prove Theorem \ref{thm:selfsimilarinstability} for any fixed pair $(\alpha,\beta)$ in the regime $0\leq\alpha\leq 1$ and $0<\beta< 3+\alpha$.
We consider the special case of radially symmetric steady temperatures $\bar{\Theta}$, called \emph{vortices}. In polar coordinates $X=Re^{i\phi}$, this corresponds to
\begin{equation}\label{eq:radiallysymmetric}
\bar{\Theta}(X)
=\bar{\Theta}(R).
\end{equation}
We work in the space $U_0$ of vortices in $L^2(\R^2)$ with zero-mean,
\begin{equation}\label{eq:U0}
U_0
:=\left\lbrace
\bar{\Theta}\in L^2
\,:\,
\bar{\Theta}(X)=\bar{\Theta}(R)\, ,\
\int_0^\infty\bar{\Theta}(R)R\dif R=0\right\rbrace.
\end{equation}
In this context, given $0\neq n\in\Z$, it is natural to seek eigenfunctions in the space of \emph{purely $n$-fold symmetric} temperatures, 
\begin{equation}\label{eq:Un}
U_n:=\{W\in L^2\,:\,W(X)=W_n(R)e^{in\phi}\}.
\end{equation}
Note that any element of $U_n$ has zero-mean. Since $U_{-n}=U_n^*$, we consider without loss of generality the case $n\in\N$.
It is easy to see that the space $U_n$ is invariant under $L_\nu$ (see Lemma \ref{lemma:Uninvariant}).

\begin{defi}\label{def:selfsimilarunstable}
We say that the vortex $\bar{\Theta}$ is \emph{unstable} if, for some $n\in\N$, there exists $0\neq W\in U_n$ satisfying $L_{0,\bar{\Theta}}W=\lambda W$ with $\Re\lambda>0$. Similarly, we say that $\bar{\Theta}$ is \emph{self-similarly unstable} if, for some $n\in\N$ and $\nu>0$, there exists $0\neq W_\nu\in U_n$ satisfying $L_{\nu,\bar{\Theta}}W_\nu=\lambda_\nu W_\nu$ with $\Re\lambda_\nu>0$.
\end{defi}

Let us recall \cite[Theorem 3.3]{CFMSSQG}, which establishes the existence of unstable vortices for the $\alpha$–SQG equation \eqref{eq:SQG:nodif} (without dissipation). Note that this is equivalent to being an unstable vortex in the sense of Definition~\ref{def:selfsimilarunstable}, since the diffusion term dissapears when $\nu = 0$.

\begin{thm}\label{thm:L}
Let $0\leq\alpha\leq 1$.
For every $n\geq 2$, there exists an unstable vortex
$\bar{\Theta}\in C_c^\infty\cap U_0$ such that the corresponding eigenfunction satisfies  $W\in C_c^\infty\cap U_n$. 
\end{thm}

In order to prove that the unstable vortex $\bar{\Theta}$ from Theorem \ref{thm:L} is also self-similarly unstable, we follow Vishik's spectral argument. This requires decomposing the operator $L_\nu$ acting on $U_n$, as
$$
L_\nu=A_\nu+C,
$$
where $(A_\nu)$ is a family of linear operators that generate contraction semigroups and possess certain continuity with respect to the parameter $\nu$, and $C$ is compact. By classical operator theory, this implies that the spectrum $\sigma(L_\nu)$ satisfies that, for any $\nu\geq 0$ and $w>0$,
$$
\sigma(L_\nu)\cap\{\Re\lambda>w\}
$$
is finite and consists of isolated eigenvalues. Now, Theorem \ref{thm:L} provides an eigenvalue $\lambda_0$ with positive real part for $\nu=0$.
Then, using the continuity with respect to the parameter $\nu$, it is possible to show that there must also be eigenvalues $\lambda_\nu$ near $\lambda_0$ for sufficiently small $\nu>0$.

\begin{prop}\label{prop:SpectralAnalysis}
Assume that the following conditions hold:
\begin{enumerate}
    \item\label{prop:SpectralAnalysis:1} Let $(A_\nu)_{\nu\geq 0}$ be a family of linear operators on some Hilbert space $H$ generating contraction semigroups. Suppose that for any fixed $\tau\geq 0$ and $W\in H$, the map 
    \begin{equation}\label{eq:mapb}
    \nu\mapsto e^{\tau A_\nu}W
    \end{equation}
    is continuous from $[0,\infty)$ to $H$.
    \item\label{prop:SpectralAnalysis:2} Let $C$ be a compact operator on $H$.
    \item\label{prop:SpectralAnalysis:3} Let $L_\nu=A_\nu+C$. Suppose there exists $\lambda_0\in\C$ with $\Re\lambda_0>0$ and $W_0\in D(L_0)$ such that
    $$
    L_0W_0=\lambda_0 W_0.
    $$
\end{enumerate}
Then, for every $\nu_0>0$, there exist $\lambda_\nu\in\C$ with $\Re\lambda_\nu>\frac{\Re\lambda_0}{2}$ and $W_\nu\in D(L_\nu)$ for some $0<\nu\leq \nu_0$ such that
$$
L_\nu W_\nu=\lambda_\nu W_\nu.
$$
\end{prop}
\begin{proof}
See \cite[Section 6.2]{CFMSSQG} for the proof.
\end{proof}

In our case, we consider $H=U_n$ and decompose the operator $L_\nu$ into
\begin{equation}\label{Lb}
L_\nu
=\nu\left(\frac{\alpha}{\beta}-1\right)
+T_\nu-\nu \Lambda^\beta+K,
\end{equation}
where 
\begin{gather*}
T_\nu\Theta=-\bar{V}_\nu\cdot\nabla\Theta
\quad\text{with}\quad
\bar{V}_\nu
=\bar{V}-\frac{\nu}{\beta}X,\\
K\Theta=-V\cdot\nabla\bar{\Theta}.
\end{gather*}
Recall that the velocities $(V,\bar{V})$ are recovered from $(\Theta,\bar{\Theta})$, respectively, through the $\alpha$-Biot-Savart law \eqref{eq:SQG:BS}. Note that the domain of $K$ is $D(K)=U_n$, the domain of $\Lambda^\beta$ is $D(\Lambda^\beta)=U_n\cap H^{\beta}$ and the domains of $T_\nu$ and $L_\nu$ are
$$
D(T_\nu)
=\{\Theta\in U_n\,:\, \mathrm{div}(\bar{V}_\nu\Theta)\in U_n\} \quad \text{and} \quad
D(L_\nu)=D(T_\nu)\cap H^\beta.
$$
Thus, the operators under consideration are closed and densely defined in $U_n$.

On the one hand, $T_\nu$ is a transport operator that satisfies the following lemma. 

\begin{lemma}\label{lemma:Tbsemigroup}
The operator $T_\nu$ generates a contraction
semigroup $(e^{\tau T_\nu})_{\tau\geq 0}$ with
\begin{equation}\label{eq:esTbbound}
\|e^{\tau T_\nu}\|_{\mathcal{L}}
	=e^{-\frac{\nu}{\beta}\tau},
\end{equation}
for all $\tau\geq 0$.
Furthermore, for any $\tau\geq 0$ and $\Theta\in U_n$, the map 
$$
\nu\mapsto e^{\tau T_\nu}\Theta
$$ 
is continuous from $[0,\infty)$ to $U_n$.
\end{lemma}
\begin{proof}
    See \cite[Lemma 6.1]{CFMSSQG} for the proof.
\end{proof}

On the other hand, $K$ is a compact operator for $0\leq\alpha<1$, while for $\alpha=1$ it can be decomposed into a skew-adjoint operator and a commutator. See the next subsections for more details.

\subsection{Case $0\leq\alpha<1$}\label{sec:selfsimilar:aalphaSQG}

In this section, we prove that $\bar{\Theta}$ is self-similarly unstable for the cases $0\leq\alpha<1$. 
To this end, we will apply Proposition \ref{prop:SpectralAnalysis} to
$$
A_\nu
=\nu\left(\frac{\alpha}{\beta}-1\right)+T_\nu-\nu\Lambda^\beta,
\qquad
C=K.
$$
Note that in \cite[Lemma 6.5]{CFMSSQG} we proved that $K$ is compact on $U_n$ for $0\leq \alpha<1$.
Then, it remains to check that $(A_\nu)_{\nu\geq0}$ satisfies condition \eqref{prop:SpectralAnalysis:1} in Proposition \ref{prop:SpectralAnalysis}.

\begin{lemma}\label{lemma:Aa<1}
The family of operators $(A_\nu)_{\nu\geq0}$ satisfies condition \eqref{prop:SpectralAnalysis:1} in Proposition \ref{prop:SpectralAnalysis}.
\end{lemma}
\begin{proof}
Since $A_\nu$ is a transport-diffusion operator (up to a time translation given by the multiple of the identity), 
it generates a strongly continuous semigroup. Namely, for any $\Theta_0\in L_n^2\cap C_c^\infty$, there exists a unique global solution 
$\Theta=e^{\tau A_\nu}\Theta_0$
to
$$\partial_\tau\Theta=A_\nu\Theta,
\qquad
\Theta|_{\tau=0}=\Theta_0.$$
By applying the identities (recall that $\mathrm{div}(\bar{V}_\nu)=-2\frac{\nu}{\beta}$)
$$
\int_{\R^2}\Theta T_\nu\Theta\dif X
=-\frac{1}{2}\int_{\R^2}\bar{V}_\nu\cdot\nabla|\Theta|^2\dif X
=-\frac{\nu}{\beta}\int_{\R^2}|\Theta|^2\dif X,
$$
and
$$
\int_{\R^2}\Theta\Lambda^\beta\Theta\dif x
=\int_{\R^2}|\Lambda^{\frac{\beta}{2}}\Theta|^2\dif x
\geq 0,
$$
the following energy estimate shows that $A_\nu$ generates a contraction semigroup
$$
\partial_\tau\int_{\R^2}|\Theta|^2\dif X
=\int_{\R^2}\Theta A_\nu\Theta\dif X
=\frac{\nu}{\beta}(\alpha-\beta-1)\int_{\R^2}|\Theta|^2\dif X\leq 0.
$$
We remark that the same inequality is obtained in  the full domain by density.
The continuity in $\nu\geq 0$ is well known for the transport-diffusion equation.
\end{proof}

\subsection{Case $\alpha=1$}\label{sec:selfsimilar:SQG}

We start by recalling the following decomposition of the operator $K$ for $\alpha=1$.

\begin{prop}\label{prop:K=S+C}
It holds that
$$
K=S+C,
$$
where $S$ is a skew-adjoint operator, and $C$ is the commutator
\begin{equation}\label{eq:commutator}
C\Theta
=\frac{1}{2}[\Lambda^{-1}\nabla^\perp,\nabla\bar{\Theta}]\Theta.
\end{equation}
Moreover, $C$ is compact in $U_n$.
\end{prop}
\begin{proof}
    See \cite[Proposition 6.3 \& Lemma 6.6]{CFMSSQG} for the proof.
\end{proof}

We now apply Proposition \ref{prop:SpectralAnalysis} to
$$
A_\nu
=\nu\left(\frac{\alpha}{\beta}-1\right)+T_\nu-\nu\Lambda^\beta+S,
\qquad
C=K-S.
$$

Firstly, we recall the stability of strongly continuous semigroups under bounded perturbations, which can be found in 
\cite[Chapter III, Bounded Perturbation Theorem]{EngellNagel00}.

\begin{prop}\label{prop:stabilitySCS}
Let $A$ be a linear operator on a Hilbert space $H$ generating a strongly continuous semigroup, and $B\in\mathcal{L}.$ Then, $A+B$ generates a strongly continuous semigroup.
\end{prop}

\begin{lemma}\label{lemma:Aa=1}
The family of operators $(A_\nu)_{\nu\geq0}$ satisfies condition \eqref{prop:SpectralAnalysis:1} in Proposition \ref{prop:SpectralAnalysis}.
\end{lemma}
\begin{proof} 
By applying Lemma \ref{lemma:Aa<1} and that $S=K-C\in\mathcal{L}$, Proposition \ref{prop:stabilitySCS} implies that $A_\nu$ generates a strongly continuous semigroup.
Similarly to the proof of Lemma  \ref{lemma:Aa<1}, but now applying that $S$ is skewadjoint, that is,
$$
\int_{\R^2}\Theta S\Theta\dif x=0,
$$
the following energy estimate on $\Theta=e^{\tau A_\nu}\Theta_0$ shows that $A_\nu$ generates a contraction semigroup
$$
\partial_\tau\int_{\R^2}|\Theta|^2\dif X
=\int_{\R^2}\Theta A_\nu\Theta\dif X
=\int_{\R^2}\Theta (A_\nu-S)\Theta\dif X
=\frac{\nu}{\beta}(\alpha-\beta-1)\int_{\R^2}|\Theta|^2\dif X\leq 0.
$$ 
The continuity in $\nu\geq 0$ holds as in Lemma \ref{lemma:Aa<1} since $K$ does not depend on $\nu$. 
\end{proof}

We have seen that, for all $0\leq\alpha\leq 1$ and $0<\beta<3+\alpha$, the requirements of Proposition \ref{prop:SpectralAnalysis} are satisfied. This guarantees the existence of $\nu>0$ for which the linearization $L_\nu$ admits an eigenpair $(\lambda_\nu,W_\nu)$ with $\Re\lambda_\nu>0$ and $W_\nu\in D(L_\nu)$. 
Finally, we check that the eigenfunction is smooth.

\begin{prop}
It holds that $W_\nu\in H^k(\R^2)$ for all $k\in\N$.
\end{prop}
\begin{proof}
The case without diffusion is treated in \cite[Proposition~6.4]{CFMSSQG}. Here, the (fractional) Laplacian allows one to deduce the regularity of the eigenfunctions by a standard bootstrapping argument.
\end{proof}

\appendix

\section{}\label{sec:appendix}

\begin{lemma}\label{lemma:Uninvariant}
For every $W=W_ne^{i n \phi}\in U_n$, it holds that
\begin{align*}
T_\nu W
&=\left(\frac{\nu}{\beta}R\partial_R-in\frac{\bar{V}_\phi}{R}\right)W_n e^{in\phi},\\
KW
&=-in\frac{V_{n,\alpha}[W_n]}{R}\partial_R\bar{\Theta}e^{in\phi},\\
\Lambda^\beta W
&=T_{n,\beta}[W_n]e^{in\phi},
\end{align*}
being
\begin{align*}
V_{n,\alpha}[W_n](R)&=C_\alpha\int_0^\infty
I_{n,\alpha}\left(\frac{R}{S}\right)W_n(S)S^{1-\alpha}\dif S,\\
T_{n,\beta}[W_n](R)&=\mathcal H_n [(\cdot)^\beta \mathcal H_n [W_n](\cdot)](R),
\end{align*}
where $C_\alpha=\frac{2^{\alpha}}{2\pi}\frac{\Gamma(1+\frac{\alpha}{2})}{\Gamma(1-\frac{\alpha}{2})}>0$,  $I_{n,\alpha}$ is the kernel
$$
I_{n,\alpha}(\sigma)
=\frac{\sigma}{n}\int_{-\pi}^{\pi}\frac{\sin(\beta)\sin(n\beta)}{|\sigma-e^{i\beta}|^{2+\alpha}}\dif\beta,
$$
and $\mathcal H_n [f]$ is the Hankel transform of order $n$ of $f$ (here $J_n$ is the Bessel function of order $n$):
$$\mathcal{H}_n [f](R)=\int_0^\infty f(\rho) J_n(\rho R)\rho\,d\rho.$$
\end{lemma}
\begin{proof}
The first two equations follow as in \cite[Corollary 2.3]{CFMSSQG}.
For the third one we proceed as follows. Using the Fourier transform
$$
\hat{f}(\xi)=\int_{\R^2}f(x)e^{-ix\cdot \xi}\dif x,
$$
and writing $X=(R,\phi)$ and $\xi=(\rho,\varphi)$ in polar coordinates, we compute
$$
\widehat W(\rho,\varphi)
=\int_0^\infty\int_0^{2\pi}
W_n(R)e^{in\phi}e^{-i R\rho\cos(\phi-\varphi)}R\,d\phi\,dR.
$$
Changing variables $\psi=\phi-\varphi$ we get
$$
\widehat W(\rho,\varphi)
=e^{in\varphi}\int_0^\infty W_n(R)R
\left(\int_0^{2\pi} e^{in\psi}e^{- iR\rho\cos\psi}\,d\psi\right)dR.
$$
Using the Jacobi–Anger expansion for $e^{- iR\rho\cos\psi}$ we get
$$
\int_0^{2\pi} e^{in\psi}e^{- iR\rho\cos\psi}\,d\psi
    = 2\pi (-i)^{n} J_n(R\rho),
$$
where $J_n$ is the Bessel function of order $n$. Thus,
$$
\widehat W(\rho,\varphi)
= 2\pi (-i)^{n} 
\left(\int_0^\infty W_n(R) J_n(R\rho)R\,dR\right) e^{in\varphi} =2\pi (-i)^{n}\mathcal H_n [W_n](\rho)e^{in\varphi}.
$$
Since $\Lambda^\beta$ acts as a Fourier multiplier with symbol $|\xi|^\beta=\rho^\beta$, we have
$$
\widehat{\Lambda^\beta W}(\rho,\varphi)
=\rho^\beta \widehat W(\rho,\varphi)
=\big(2\pi (-i)^{n}\rho^\beta \mathcal H_n [W_n](\rho)\big)e^{in\varphi}.
$$
Applying the inverse Fourier transform, which again separates variables, we conclude
$$
\Lambda^\beta W(R,\phi)
=\left(\int_0^\infty \rho^\beta \mathcal H_n [W_n](\rho) J_n(R\rho)\rho\,d\rho\right) e^{in\phi}=\mathcal H_n [\rho^\beta H_n [W_n](\rho)](R) e^{in\phi}.
$$
Hence $\Lambda^\beta W$ is of the form (radial function)$\times e^{in\phi}$,
and therefore $\Lambda^\beta$ preserves $n$-fold symmetry.
\end{proof}

\begin{lemma}\label{lemma:tellme}
    For any $s\in\mathbb{R}$, $h\in H^s$ and $\lambda>0$ we have
    $$\Lambda^s (h_\lambda)=\lambda^s(\Lambda^sh)_\lambda$$
    where $h_\lambda(x)=h(\lambda x).$
\end{lemma}
\begin{proof}
We denote the Fourier transform as before
$$
\hat{f}(\xi)=\int_{\R^2}f(x)e^{-ix\cdot \xi}\dif x.
$$
Note that the operator $\Lambda^s$ acts as a Fourier multiplier with symbol $|\xi|^s$:
$$
\widehat{\Lambda^s f}(\xi) = |\xi|^s \hat{f}(\xi).
$$
It is easy to check that
$$
\widehat{f_\lambda}(\xi)=\frac{1}{\lambda^{2}} \hat{f}\left(\frac{\xi}{\lambda}\right)
$$
and, therefore,
$$\widehat{\Lambda^s f_\lambda}(\xi) = |\xi|^s \frac{1}{\lambda^{2}} \hat{f}\left(\frac{\xi}{\lambda}\right).$$
Finally, we return to the spatial domain using the inverse formula
\begin{align*}
\Lambda^s f_\lambda(x) &= \frac{1}{(2\pi)^2} \int_{\mathbb{R}^2} \left[ |\xi|^s \frac{1}{\lambda^{2}} \hat{f}\left(\frac{\xi}{\lambda}\right) \right] e^{i x \cdot \xi} \, d\xi= \frac{1}{(2\pi)^2} \int_{\mathbb{R}^2} \lambda^s |\eta|^s \frac{1}{\lambda^{2}} \hat{f}(\eta) e^{i (\lambda x) \cdot \eta} \lambda^2 \, d\eta\\
&= \lambda^s \left[ \frac{1}{(2\pi)^2} \int_{\mathbb{R}^2} |\eta|^s \hat{f}(\eta) e^{i (\lambda x) \cdot \eta} \, d\eta \right]= \lambda^s \left[ \frac{1}{(2\pi)^2} \int_{\mathbb{R}^2} \widehat{\Lambda^s f}(\eta) e^{i (\lambda x) \cdot \eta} \, d\eta \right] \\
&= \lambda^s (\Lambda^s f)(\lambda x)=\lambda^s (\Lambda^s f)_\lambda(x).
\end{align*}
This concludes the proof.
\end{proof}


\subsection*{Acknowledgments}

The authors thank \'Angel Castro and Daniel Faraco for their valuable comments during the preparation of this work.

The authors acknowledge support by the Generalitat Valenciana through the call ``Subvenciones a grupos de investigación emergentes", project CIGE/2024/115.

F.M.~acknowledges support from grant RYC2023-045748-I, and from grants PID2024-158418NB-I00 and PID2024-158664NB-C2 funded by MCIN/AEI/10.13039/501100011033.

M.S.~acknowledges support from grant PID2022-136589NB-I00 funded by MCIN/AEI/10.13039/ 501100011033 and by ERDF ({\it A way of making Europe}). 

\bibliographystyle{abbrv}
\bibliography{Nonuniqueness_LPS_force}

\begin{flushleft}
	\quad\\
	Francisco Mengual\\
	\textsc{Universidad de Sevilla, IMUS\\
		41012 Sevilla, Spain}\\
	\emph{E-mail address:} fmengual@us.es
\end{flushleft}

\begin{flushleft}
	\bigskip
	Marcos Solera\\
	\textsc{Departament d’An\`alisi Matem\`atica, Universitat de Val\`encia\\
		46100 Burjassot, Spain}\\
	\emph{E-mail address:} marcos.solera@uv.es
\end{flushleft}

\end{document}